\documentclass[12pt,a4paper]{article}

\usepackage{lineno,hyperref}
\modulolinenumbers[5]
\usepackage{amssymb,amsfonts,amsmath}
\usepackage[all,arc]{xy}
\usepackage{fullpage}
\usepackage{xspace}
\usepackage{amssymb, mathrsfs}
\usepackage{amsmath, graphicx, amsthm}
\usepackage{times}
\usepackage{pgf}
\usepackage{tikz}
\usepackage{tikz-cd}
\usetikzlibrary{matrix,arrows,decorations.pathmorphing}
\usepackage{enumerate}
\usepackage{hyperref}
\hypersetup{
    colorlinks=true,
    linkcolor=blue,
    filecolor=magenta,      
    urlcolor=cyan,
    citecolor=red,
    bookmarks=true
   }
\usepackage{cleveref}
\usepackage{mathtools}
\usepackage{latexsym}
\usepackage{textcomp}
\usepackage{amscd}
\usepackage{amsxtra}
\usepackage{stmaryrd}
\usepackage{mathrsfs}
\usepackage{todonotes}
\usepackage{ulem}

\newtheorem{sbthm}[subsubsection]{Theorem}

\newtheorem{sbcor}[subsubsection]{Corollary}
\newtheorem{sblem}[subsubsection]{Lemma}

\newenvironment{pf}{\proof[\proofname]}{\endproof}

\theoremstyle{definition}

\newtheorem{sbdefn}[subsubsection]{Definition}
\newtheorem{sbeg}[subsubsection]{Example}
\newtheorem{sbrem}[subsubsection]{Remark}

\theoremstyle{definition}

\theoremstyle{remark}

\newcommand{\tit}{\textit}

\newcommand{\cH}{{\mathcal{H}}}

\newcommand{\sH}{{\mathscr{H}}}

\newcommand{\al}{\alpha}
\newcommand{\be}{\beta}
\newcommand{\ga}{\gamma}

\newcommand{\si}{\sigma}

\newcommand{\ze}{\zeta}

\newcommand{\Z}{\mathbb Z}
\newcommand{\N}{\mathbb N}

\newcommand{\sub}{\subset}

\newcommand{\nin}{\notin}

\newcommand{\ol}{\overline}

\newcommand{\dis}{\displaystyle}

\newcommand{\Rar}{\Rightarrow}
\newcommand{\lb}{\left(}
\newcommand{\rb}{\right)}
\newcommand{\F}{\mathbb{F}}

\newcommand{\js} {\mathcal{J}_{\si}}

\newcommand{\z}{\mathfrak{z}}
\newcommand{\m}{\mathfrak{m}}

\DeclareMathOperator{\tr}{Tr}

\DeclareMathOperator{\ch}{char}

\DeclareMathOperator{\Gal}{Gal}
\DeclareMathOperator{\Sw}{Sw}

\DeclareMathOperator{\Tr}{Tr}
\DeclareMathOperator{\KSw}{KSw}

\makeatletter
\let\c@equation\c@thm
\makeatother
\numberwithin{equation}{section}

\begin{document}

\title{Degree $p$ Extensions of Arbitrary Valuation Rings and ``Best $f$"}

%% Group authors per affiliation:
\author{Vaidehee Thatte}

\date{}
\maketitle

\begin{abstract} 

We  prove the explicit characterization of the so-called ``best f" for degree $p$ Artin-Schreier and degree $p$ Kummer extensions of Henselian valuation rings in residue characteristic $p$. This characterization is mentioned briefly in \cite{V1, V2}. Existence of best $f$ is closely related to the defect of such extensions and this  characterization plays a crucial role in understanding their intricate structure.
We also treat degree $p$ Artin-Schreier defect extensions of higher rank valuation rings, extending the results in \cite{V1}, and thus completing the study of degree $p$ extensions that are the building blocks of the general theory. 

\end{abstract}

\tableofcontents

\section{Introduction}\label{s:in}

\subsection{Outline} \indent We will now explain the notation (\ref{N1}), discuss the obstacles in the general case and the notion of defect (\ref{obstacles}), and explain the set-up for degree $p$ extensions (\ref{sup}) that will be used in the rest of this paper.
In \S\ref{swf}, we review the classical Swan conductor (\ref{csw}) as well as Kato's Swan conductor and the notion of ``best $f$" for discrete valuation rings with possibly imperfect residue fields (\ref{Ksw}).
In \S\ref{chbest}, we state and prove the characterization of best $f$ in the general case for ``defectless'' degree $p$ Artin-Schreier extensions ({\bf \Cref{best f}}) and Kummer extensions ({\bf \Cref{best h}}), in residue characteristic $p$. Finally in \S \ref{high}, we discuss the case of  Artin-Schreier extensions with higher rank valuations and non-trivial defect. In particular, we prove \cite[Theorem 0.3]{V1}, without any restrictions on the rank of the valuation ({\bf \Cref{hn}}).  This allows us to extend all the key results of \cite{V1} to the higher rank defect case.

\subsection{Notation}\label{N1} 
 Let $A$ be a Henselian valuation ring that is not a field, and let $K$ be the field of fractions of $A$. Suppose $\Gamma_A$ denotes the value group $K^\times/A^\times$ of $A$ and $v_A$ denotes the corresponding  additive valuation on $K$.   The valuation ring  $A$ consists of elements with non-negative valuation and is partitioned into the set of  its units $A^{\times}$ (elements of valuation $0$) and its unique maximal ideal  ${\mathfrak m}_A$ (elements with strictly positive valuation). Let $k$ denote the residue field $A/\m_A$ of $A$.   For a finite Galois extension $L$ of $K$, we denote the integral closure of $A$ in $L$ by $B$. Since $A$ is Henselian, $v_A$ extends uniquely to a valuation $v_B$ on  $L$, making $B$ a Henselian valuation ring as well. The index $[\Gamma_B:\Gamma_A]$ is called the {\it ramification index} and is denoted by $e_{L/K}$. Let $l$ denote the residue field of $L$. The degree $[l:k]$ of the residue extension is called the {\it inertia degree} and is denoted by $f_{L/K}$.

\subsection{Classical versus general case and the defect}\label{obstacles}
In the classical case, we consider a complete discrete valuation ring $A$ with a perfect residue field $k$. In the general case, we consider Henselian valuation rings and remove all the other restrictions on the valuation. Some of the immediate challenges we face due to this change are discussed below. We also recall the definition of the defect in the third point.

\begin{enumerate}
\item The ideals of $A$ and $B$ need not be principal. A non-principal ideal cannot be finitely generated.

\item The ring  $B$ may not be finitely generated as an $A$-algebra; classically it is monogenic. 

\item  There is a positive integer $d_{L/K}$, called the \tit{defect} of the extension, such that   $$[L:K]=d_{L/K}e_{L/K}f_{L/K}.$$ The extension $L/K$ is called \tit{defectless} if the defect is trivial, that is, if $d_{L/K}=1$.\\ 
 When $\ch k = 0$, there is no defect.\\ When  $\ch k = p>0$, the defect is a non-negative integer power of $p$, and it could be non-trivial.\\ The classical case is  defectless, regardless of the residue characteristic. 
 \end{enumerate}
 One may see \cite{Ku} for further discussion on the defect.
\begin{sbeg}[Classical]\label{clex}

Let $K= \F_p((X))$ and let $L=K(\al)$ be the extension generated by the roots of $\al^p-\al=1/X^n$, where $n$ is a positive integer co-prime to $p$. In this case, $e_{L/K}=p$ and $d_{L/K}=1=f_{L/K}$. In particular, the extension is defectless.
\end{sbeg}
\begin{sbeg}[Defect]\label{defex} Let $\dis K=\cup_{r \in \Z_{\geq 0}} \F_p((X))(X^{1/p^r})$ and let $L=K(\al)$ be the extension generated by the roots of $\dis \al^p-\al=1/X$. In this case, $d_{L/K}=p$ and $e_{L/K}=1=f_{L/K}$. This extension is  ``almost unramified". A family of much worse (and hence, more interesting) examples created by a sequence of blow-ups is available in the appendix of \cite{V1}.
\end{sbeg}

\subsection{Degree $p$ set-up}\label{sup}
 In  this entire paper (unless mentioned otherwise), we focus on degree $p$ extensions $L/K$ for a prime $p$, such that the residue characteristic $\ch k$ is $p$. 
 
 Then $L/K$ is either \tit{unramified} ($l/k$ is separable of degree $p$), \tit{wild} (ramification index is $p$), \tit{ferocious} ($l/k$  is purely inseparable of degree $p$), or has non-trivial defect. The Galois group $\Gal(L/K)$  is cyclic of order $p$, with any non-trivial element acting as its generator, fix one such generator  $\si$.\\
Depending on $\ch K$, we have the following two cases.
\\
{\bf Equal characteristic case ($\ch K =p$):}  
In this case, $L=K(\al)$ is an Artin-Schreier extension where $\al \in L \backslash K$ satisfies the polynomial $\al^p-\al=f$ for some $f \in K^{\times}$. 
\\
{\bf Mixed characteristic case ($\ch K =0$ and $p \in \m_A$):}
 If $K$ contains a primitive $p$-th root $\ze$ of unity, $L=K(\al)$ is a Kummer extension defined
by $\al^p =h$ for some $h \in K^{\times} $where $\al \in L \backslash K$.
The non-Kummer case ($\ze \nin K$) is reduced to the Kummer case (see  \cite[\S7]{V2}).

\section{Swan conductor and ``best $f$" for DVRs}\label{swf}

\subsection{Classical Swan conductor $\Sw_{L/K}$}\label{csw} In the classical case of complete discrete valuation rings with perfect residue fields, we have two important  invariants of wild ramification, namely the logarithmic Lefschetz
number $j(\si)$ and the \tit{Swan conductor  of $L/K$} defined as follows. Note that  $j(\si)$ is independent of the choice of the generator $\si$ of $\Gal(L/K)$.
 \begin{sbdefn}\label{swan}
$$j(\si):=\min\left\{v_B\lb\frac{\si(a)}{a} -1\rb\mid a \in L^{\times} \right\}~~ \text{ and } ~~\Sw_{L/K}:=\frac{j(\si) p}{e_{L/K}}.$$
\end{sbdefn}

\begin{sbrem}Both can be defined for a degree $n$ extension \cite{S}, we have restricted ourselves to degree $p$ in this paper. Their non-logarithmic analogs are omitted as well, since we do not use them in this paper.\end{sbrem}

\subsection{Kato's Swan conductor $\Sw_{L/K}$ for DVRs and ``best $f$''}\label{Ksw}
{\bf Motivating example:} Consider \cref{clex} where $K= \F_p((X))$ and $L=K(\al)$ is the extension generated by the Artin-Schreier equation $\al^p-\al=1/X^n; p \nmid n \in \N$. In this case, one can check that $\Sw_{L/K}$ is in fact $n$, the same as the order of the pole of $1/X^n$. 

Many other Artin-Schreier polynomials in $K[T]$ give rise to the same extension (see \ref{allf}). In particular, any polynomial of the form $\dis T^p-T=(1/X^n)+g^p-g; g\in K$ generates the same extension $L$ since its roots are of the form $\al+g$. One such example is $\dis T^p-T=1/X^{np}$ where $\dis g=1/X^n$. However, $1/X^n$ (when $g=0$) is the one minimizing the order of the pole.

\begin{sbdefn}\label{ksw} Kato gave the following definition of the Swan conductor $\KSw_{L/K}$ for Artin-Schreier extensions of discrete valuation rings with possibly imperfect residue fields. 
 $$\KSw_{L/K}:=\min \{v_A(1/h) \mid   \text{ and the solutions of
the equation } T^p-T=h \in K^{\times} \text{ generate } L/K \}.$$ 
An element $f$ of $K$ which attains this minimum is called \tit{best $f$}. The existence of best $f$ is guaranteed since the value group is $\Z$. \end{sbdefn}

\noindent In the classical case, $\KSw$ coincides with $\Sw$. This can be seen as follows. Let $\al$ be an  Artin-Schreier generator corresponding to best $f$ and let $N_{L/K}$ denote the norm map. Then $j(\si)=v_B(1/\al)$, and by \cref{swan}, we have the following.
\subsubsection{Connection via $N_{L/K}$:}\label{Kswan} $$\KSw_{L/K}= v_A(1/f)=v_A(N_{L/K}(1/\al))=\frac{v_B(1/\al) p}{e_{L/K}}=\frac{j(\si) p}{e_{L/K}}= \Sw_{L/K}$$

\noindent In the above equation, $v_A(1/f)$ gives us a $K$-side description of $\Sw$ while $(v_B(1/\al) p)/e_{L/K}$ gives an $L$-side description. The two are connected via the norm map. We generalize this connection later in \Cref{hn}.

\begin{sbrem} The mixed characteristic analogue for Kummer extensions is defined in a similar fashion and also satisfies this.\end{sbrem}

\section{Characterization of Best $f$}\label{chbest}
Now we return to  the general case and use the notation in \ref{N1}, \ref{sup}.
\\\\
As we will see in \Cref{best f} and \Cref{best h}, when the extension is defectless, we can tackle the first two obstacles in the general case (\ref{obstacles}) by using best $f$ and best $h$, respectively. Explicit descriptions of these best generators are also very useful in understanding the structure of the extensions. We note that such characterization has appeared in \cite[3.6]{K2} and \cite[Lemma (2-16)]{OH} for  discrete valuation rings.
\subsection{Cyclic extensions of prime degree.} 
In \cite{V1} we proved the following lemmas for $L/K$ as in \ref{sup} assuming $\ch K=p$, but we did not use this in the proofs (we just needed $\ch k = p$). Hence, the same results are available in the mixed characteristic as well. The statements (using the notation in this paper) are as follows, and they correspond to Lemma 1.11 and Lemma 1.12 of \cite{V1}, respectively.

\begin{sblem}\label{cyclic}
If $L/K$ is ramified and defectless, then we have two cases:

\begin{enumerate}[(a)]
\item Order of $\Gamma_B/ \Gamma_A$  is $p$  and it is
generated by $v_B(\mu)$ for some $\mu \in L^{\times}$.
\item There is some $\mu \in B^{\times}$ such that the residue
extension $l/k$ is purely inseparable of
degree $p$, generated by the residue class of $\mu$.
\end{enumerate}
\end{sblem} 
\noindent For such an extension $L/K$, these two cases correspond to the types wild ($e_{L/K}=p$) and ferocious (ramified and $f_{L/K}=p$), respectively. 
\\
In \cite[Lemma 1.11(b)]{V1} we state $\mu \in L^\times$. But since we consider its residue class, it needs to come from $B$, so it is in fact an element of $B^{\times}$.
\begin{sblem}\label{mu}
Let $L/K \text{ and } \mu$ be as in \Cref{cyclic} and $x_i \in K$ for all $0
\leq i \leq p-1$ . Then

$$ \sum_{i=0}^{p-1} x_i \mu^{i} \in B \iff x_i \mu^i \in B \text{ for all } i.$$.
\end{sblem}

\begin{sbrem}\label{muval}
    Let $L/K, \mu$ be as in \Cref{cyclic}$(b)$. 
    
    Then by \Cref{mu}, $B=A[\mu]$.
    Since $e_{L/K}=1$, (let $v$ denote the valuation on both $L$ and $K$) for all $x \in \in L^\times$ we have 
    $v(x)= \sup\{v(a) \mid a \in K^\times, a^{-1}x \in B\}.$
  
Let $\dis x=\sum_{i=0}^{p-1} x_i \mu^{i}$; $x_i \in K$ for $0 \leq i \leq p-1$.  Using  $B=A[\mu]$, we have $$v(x)=\min_{0\leq i\leq p-1} v(x_i \mu^{i})=\min_{0\leq i\leq p-1} v(x_i).$$
\end{sbrem}
\subsection{Equal characteristic case: Artin-Schreier Extensions}\label{eq}

\noindent For this subsection we assume that $\ch K= \ch k = p$ and that $L=K(\alpha)$ is a non-trivial Artin-Schreier extension generated by the equation $\alpha^p-\alpha = f \in K$. By slight abuse of terminology, we will say that $f$ generates $L/K$.

\subsubsection{}\label{allf} Given any two $f$ and $f'$ generating the same Artin-Schreier extension $L/K$, there exists  $h \in K$ and an integer $i \in \{1, \cdots, p-1\}$ such that $f'=i(f+h^p-h)$.
Note that since $A$ is Henselian, if we can find an  $f \in \m_A$, then the resulting Artin-Schreier extension $L/K$ is trivial. Thus, we may assume that any such $f \in K \backslash \m_A$.
\begin{sbdefn}
    
An element $f \in K$ generating $L/K$ as above is called \tit{best} if and only if $$  v_A\lb\frac{1}{f}\rb=\inf_{h \in K^{\times}, 1 \leq i \leq p-1 } v_A\lb\frac{1}{f+h^p-h}\rb$$
\end{sbdefn}

If  $f$ is not best, then a ``better'' generator with higher valuation exists. That is,  there exists some $h \in K$ such that $v_A(f+h^p-h)>v_A(f)$. We observe that this is only possible when $v_A(f)=v_A(h^p-h)=pv_A(h)<0$.

\begin{sbthm}\label{best f} The following are equivalent for a non-trivial extension $L/K$:

\begin{enumerate}[(a)]
\item  Best $f$ exists.
\item There exists $f$ satisfying exactly one of the properties (i) -- (iii).
\begin{enumerate}[(i)]
\item $f\notin A$ and the class of $f$ in $K^\times/A^\times$ is not a $p$-th power, that is, the valuation $v_A(f)$ is not divisible by $p$.

\item $f\notin A$ and $f=ug^{-p}$ for some $0 \neq g\in \m_A$ and $u\in A^{\times}$ such that the residue class of $u$ is not a $p$-th power in $k$. 

\item $f\in A^{\times}$ and the residue class $\bar{f}$ of $f$ does not belong to $\{x^p-x\; \mid \; x\in k\}$. 
\end{enumerate}
\item The extension $L/K$ is defectless.
\end{enumerate}
\end{sbthm}

\begin{pf}
We will prove this in two stages - $(a) \iff (b)$ and $(b) \iff (c)$.

Suppose that best $f$ exists and  $(b)$ fails. This best $f$ must be of the form $(u^p+a)g^{-p}$ for some unit $u$ of $A$ and $a \in \m_A$. Consider $f'=f+h^p-h$ where $h=(-u)g^{-1}$. Then $v_A(f')=v_A(a+ug^{p-1})-pv_A(g)>v_A(f)=-pv_A(g)$ leading to the contradiction that $f'$ is better than $f$. 

Conversely, we will show that any $f$ satisfying one of the properties (i) -- (iii) is in fact best. This is clear in case (i) since $p \nmid v_A(f)$. It is also clear in case (iii) since any $f'$ generating $L/K$ satisfies $v_A(f')\leq 0=v_A(f)$. 

Now consider $f$ satisfying (ii) and assume that $f'=i(f+h^p-h)$ is better than $f$, that is, $v_A(f')>v_A(f)$. Therefore, $v_A(h)=-v_A(g)$ and $g^p(ug^{-p}+h^p-h) = u +(gh)^p-g^{p-1} (gh)$ is an element of $ \m_A$. This contradicts (ii) since the residue class of $u$ is a $p$ the power of the residue class of $gh$ in $k$. This proves the equivalence of $(a)$ and $(b)$.
\\

Assume that $(b)$ holds  and let $\al$ be an Artin-Schreier generator corresponding to such an $f$, that is, $\al^p-\al=f$. In the case (i), the ramification index of $L/K$ is $p$ and the quotient $v_B(L^{\times})/v_B(K^{\times})$ is generated by $v_B(\al)$. In the case (ii), the residue  extension is purely inseparable of degree $p$, generated by the the residue class of $\al g$, the $p$-th root in $l$ of the residue class of $u$.  In the case (iii), $L/K$ is unramified and the residue extension is separable of degree $p$. Thus, the extension $L/K$ is defectless.

Next, suppose that  $(c)$ is true but assume to the contrary that $(b)$ fails. Consequently, $(a)$ fails and best $f$ does not exist. Since $L/K$ is defectless, the ramification index $e_{L/K}=p$ or the inertia degree $[l:k]=p$.
Since $(b)$(i) is not true, any such $f \in K \backslash \m_A$ must have $p$-divisible valuation. Therefore, the ramification index is trivial and we must have $[l:k]=p$. 

Since $e_{L/K}=1$, let us denote by $v$ the valuation on both $K$ and $L$. Begin with an Artin-Schreier equation $\al^p-\al=f$ generating $L/K$. All the  Artin-Schreier generators of $L/K$ are of the form $i(\al+h)$ for some $h \in K$ and an integer $i \in \{1, \cdots, p-1\}$. We have $v(\al)=v(f)$, and since $(a)$ fails, $\sup v(i(\al+h))$ over all the generators is not attained by any generator.

Since $(b)$(iii) is not true, $L/K$ is not unramified. We can apply lemmas \ref{cyclic}(b) and \ref{mu}; let $\mu \in B^{\times}$ be as in these results. The elements $1, \mu, \cdots, \mu^{p-1}$ form a basis of $L/K$ while their residue classes form a basis of the residue extension $l/k$.

We can write the above $\alpha$ as $\dis \alpha=\sum_{j=0}^{p-1}x_j \mu^j$, where $x_j \in K$ for all $j$, not all zero. 
 Then  for  any $h \in K$ and any integer $i \in \{1, \cdots, p-1\}$, by \cref{muval}, we have 

$$v(i(\al+h))=v(\al+h)=v((x_0+h)+\sum_{j\neq 0} x_j\mu^j)=\min\{v(x_0+h),  v(x_j\mu^j)_{j\neq 0}\}.$$

\noindent Thus, $\dis \sup_{i, h} v(i(\al+h)) \leq \min_{j\neq 0} \{v(x_j\mu^j)\} = \min_{j\neq 0} \{v(x_j)\}$. 
\\This bound is attained when we consider $h=-x_0$ and we reach a contradiction.
\\
This proves the equivalence of $(b)$ and $(c)$, and concludes the proof.
\end{pf}
\noindent The following corollary was obtained in the proof of $(b) \implies (a)$.
\begin{sbcor}\label{bestase} Any $f$ satisfying (i) -- (iii) is best.
\end{sbcor}
\subsection{Mixed characteristic case: Kummer extensions}\label{mx}
For this subsection we assume that $\ch K= 0,~ \ch k = p$. 
We have an analoguous result characterizing ``best'' generators in the mixed characteristic case when the base field $K$ contains a primitive $p$-the root $\ze:=\ze_p$ of unity.

Let $L/K$ be a non-trivial Kummer extension given by the equation $\alpha^p = h \in K$. Again, by slight abuse of terminology, we will say that $h$ generates $L/K$. 

Let $\dis \z:= \ze -1$ and $\dis e':=\frac{v_A(p)}{p-1}=v_A(\z).$

\subsubsection{}\label{allh} Given any two equations $\al^p=h$ and $\al'^p=h'$ generating the same Kummer extension $L/K$, there exists  $g \in K$ and an integer $i \in \{1, \cdots, p-1\}$ such that $h'=h^ig^p$, by Kummer theory.\\  
\\ 
 Since $L/K$ is non-trivial, any such $h$ must satisfy $v_A(h-1) \leq v_A(\z^p)=e'p$.
\\ 
For if $\al^p-1=h-1 = a\z^p$ for some $a \in \m_A$, then consider $\dis \be:=(\al-1)/\z \in L\backslash K$.\\ We have  $\dis a= 
((1+\z \be)^p-1)/\z^p \equiv \be^p-\be \mod \z$. Thus, $\be^p-\be \equiv 0 \mod \m_B$. Because $K$ is Henselian,  $\be$ and consequently $\al$ must be elements of $K$, contradicting non-triviality of $L/K$.
\\\\
Given $h, h'$ generating $L/K$, we observe that 
$\dis v_A(\z^p/(h-1)) \leq v_A(\z^p/(h'-1))$ if and only if $\dis v_A(h'-1) \leq v_A(h-1)$. If $h \in A$ and $h' \in K \backslash A$, then this inequality is trivially true. \\\\
Let $g \in K^{\times}, 1 \leq i \leq p-1$ be such that $h'=h^ig^p$. If $h \in A^{\times}$ and $g \nin A^{\times}$, then $\dis v_A(h'-1)\leq 0 \leq v_A(h-1)$.

\begin{sbdefn}
    
An element $h \in K$ generating $L/K$ as a Kummer extension is called \tit{best} if and only if $$  v_A\lb\frac{\z^p}{h-1}\rb=\inf_{g \in K^{\times}, 1 \leq i \leq p-1 } v_A\lb\frac{\z^p}{h^ig^p-1}\rb$$
equivalently, 
$$v_A(h-1)=\sup_{g \in K^{\times}, 1 \leq i \leq p-1 } v_A\lb h^ig^p-1\rb.$$
\end{sbdefn}
\begin{sbthm}\label{best h} The following are equivalent for a non-trivial Kummer extension $L/K$:

\begin{enumerate}[(a)]
\item  Best $h$ exists.
\item There exists $h \in A$ satisfying exactly one of the properties (i) -- (v).
\begin{enumerate}[(i)]
\item $p \nmid v_A(h) $.
\item $h=u$.
\item $h=1+ct;~ 0 <v_A(t)<e'p, p \nmid v_A(t) $.
\item $h=1+us^p;~ 0<v_A(s)<e'$. 
\item $h=1+c \z^p;~ \ol{c} \nin \{x^p-x \mid x \in k \}$.
\end{enumerate}
where $s, t \in \m_{A}, u, c \in A^{\times}, \ol{u} \nin k^p$.

\item The extension $L/K$ is defectless.
\end{enumerate}
\end{sbthm}
\begin{pf}
We will prove this in two stages - $(a) \iff (b)$ and $(b) \iff (c)$. Unless mentioned otherwise, $i, g$ in this proof satisfy $1 \leq i \leq p-1$ and $g \in K^{\times}$. When we compare two elements $h, h'=h^ig^p$, we may restrict our attention to the case $h, h' \in A$.\\

Suppose that best $h$ exists, but $h$ itself does not satisfy any of the properties (i) -- (v). Let  $\al^p=h$ and $\be:=(\al-1)/\z$. If $v_A(h)>0$, then $p \mid v_A(h)$ and we can replace $h$ with a unit $u$ that satisfies (ii). If $v_A(h)=v_A(h-1)=0$, then $L/K$ is a trivial extension unless $h$ satisfies (ii). This leaves us with the case $v_A(h-1)>0$. 

If $h-1=c \z^p$ with $c \in A^{\times}, \ol{c} \in \{x^p-x \mid x \in k \}$, then we have $c \equiv \be^p-\be \mod \z$, contradicting non-triviality of $L/K$.
Thus, the only possibility is that $h-1=(\lambda^p+a)s^p$ for some $\lambda \in A^{\times}, a,s \in \m_A; v_A(s)<v_A(\z)$.  Let $g=1/(\lambda s+1) \in A^{\times}$. 

For some $x \in A$, we have
$$\frac{hg^p-1}{s^p}=\frac{(\lambda^p+a)s^p+1-(\lambda s+1)^p}{s^p(\lambda s+1)^p}=\frac{as^p+(p \lambda s) x}{s^p(\lambda s+1)^p}=\frac{a+\lambda x (ps/s^p)}{(\lambda s+1)^p}.$$

\noindent Since $a, (ps/s^p) \in \m_A$ and $g$ is a unit, $v_A(s^p)=v_A(h-1)< v_A(hg^p-1)$, contradicting the assumption that $h$ is best.
\\\\
For the converse, we will show that any $h \in A$ satisfying one of  (i) -- (v) is best.
\\
Note that $h'=h^ig^p$ is ``better" than $h$ if and only if $v_A(h^ig^p-1)>v_A(h-1)$.
Since $L/K$ is non-trivial, any $h$ satisfying (v) must be best. We will consider the remaining cases now.
\begin{enumerate}[(i)]
\item  If $h'=h^ig^p$ is better, then $v_A(h^ig^p-1)>v_A(h-1)=0$. This is only possible if $h^ig^p$ is a unit. Since $p \nmid v_A(h)$, $h$ must be best.
\item Let $h=u  \in A^{\times}, \ol{u} \nin k^p$.  If $h'=u^ig^p$ is better, then $u^ig^p \equiv 1 \mod \m_A$. This is impossible since $\ol{u} \nin k^p$.
\end{enumerate}
Assume that $h$ satisfies (iii) or (iv) and that $h'=h^ig^p$ is better than $h$. Since $h$ is a unit, $g$ must also be a unit for this to be true. 

Writing $\dis h^ig^p-1=(h-1+1)^ig^p-1=(h-1)^2x+i(h-1)g^p+g^p-1; x \in A$ we see that $\dis v_A(h^ig^p-1)>v_A(h-1)$ if and only if $\dis v\lb ig^p+\frac{g^p-1}{h-1}\rb>0$. Since $ ig^p$ is a unit, this can only happen if $\dis \frac{g^p-1}{h-1}$ is a unit as well, that is, if $\dis v_A(g^p-1)=v_A(h-1)<v_A(\z^p)$.\\\\
If $v_A(g-1) \geq v_A(\z)$, then
$$v_A(g^p-1)=v_A(((g-1)+1)^p -1) = v_A((g-1)^p+p(g-1)^{p-1}+\cdots+p(g-1))\geq v_A(\z^p).$$ Thus, we must have $v_A(g-1)<v_A(\z)$.\\ Consequently, $v_A(g^p-1)=v_A((g-1)^p)=pv_A(g-1)=v_A(h-1)$. As this requires $v_A(h-1)$ to be $p$-divisible, such an $h$ cannot satisfy (iii) (in other words, any $h$ satisfying (iii) is best). 
\\\\
Let $\dis h-1=us^p; s \in \m_A, u  \in A^{\times}, \ol{u} \nin k^p$. Since $\dis v\lb ig^p+\frac{g^p-1}{h-1}\rb>0$, there exists $\dis a \in \m_A$ such that 
$\dis ig^p+\frac{g^p-1}{us^p}=a$. Thus, $\dis u=\frac{au}{ig^p}-\frac{(g^p-1)}{is^pg^p}$. \\\\ As $p \in \m_A$ and $p \mid i^p-i$, modulo $\m_A$ we have
$$u\equiv (-1) \frac{(g^p-1)}{is^pg^p} \equiv \lb\frac{(-1)(g-1)}{isg}\rb^p$$
contradicting $\ol{u} \nin k^p$. This concludes the proof of $(b) \implies (a)$.
\\
\\
We will now prove $(b) \implies (c)$.
\\
In the cases (i) and (iii), $v_B(L^{\times})/v_B(K^{\times})$ is generated by images of $v_B(\al)$ and $v_B(\al-1)$,  respectively.  Thus, the ramification index is $p$ and the extension is defectless. 
\\
We next observe that in the remaining cases the inertia degree is $p$ and hence, the extension is defectless. 
\\
In the case (ii), the residue extension is purely inseparable of degree $p$, generated by the class of $u$. Similarly in the case (iv), it is generated by the class of $(\al-1)/s$. 
\\
In the case (v), write $\dis \al=\z \be+1$. Hence, $\dis c\equiv \be^p-\be \mod (\z)$ and $l/k$ is separable, generated by the roots $\dis \{\ol{\be}+j\}_{0 \leq j \leq p-1}$ of the polynomial $\dis T^p-T-\ol{c} \in k[T]$.
\\

Next, we prove $(c) \implies (b)$. The proof is similar to that in the equal characteristic case. 
Let $\al^p=h$ generate $L/K$. Assume that the extension is defectless, but $(b)$ fails. Consequently, $(a)$ fails and $\dis \sup_{g  \in K^{\times}, i}v_A(h^ig^p-1)$ is not attained.
 Note that 
 $$ \sup_{g  \in K^{\times}, i}v_A(h^ig^p-1)=\sup_{g  \in K^{\times}, i,  h^ig^p \in A^{\times}}v_A(h^ig^p-1).$$ Thus, the $K$-valuation $v_A(h^ig^p-1)$ must always be $p$-divisible.  In this case, we see that the ramification index is $1$, and therefore, the inertia degree must be $p$. We'll denote by $v$ the valuations on both $L$ and $K$. 
 \\
 Observe that
$\dis \sup_{g  \in K^{\times}, i}v(\al^ig-1)$ is not attained. Hence, for some fixed $i$, $\dis \sup_{g  \in K^{\times}}v(\al^ig-1)$ is not attained either. We will use such an $i$ later in the proof.
\\\\
Note that the extension $L/K$ must be ramified. Once again, we can apply the lemmas \ref{cyclic}(b) and \ref{mu}. Let $\mu \in B^{\times}$
be as in these results.
 Again, we see that $1, \mu,  \cdots, \mu^{p-1
 }$ form a $K$-basis of $L$. We use \cref{muval}, as before.
\\
 Let $i$ be such that $\sup_{g  \in K^{\times}}v(\al^ig-1)$ is not attained and write 
 $$\al^i-1=\sum_j x_j\mu^j=x_0+\sum_{j \neq 0} x_j\mu^j.$$ 
 
 \noindent Since $v(\al^i-1)>0$ (we can assume this since the $\sup$ is not attained), $x_j \in \m_A$ for all $j$.\\ 
For $g \in K^{\times}$, we can write

$$\al^i g -1=(\al^i-1)g+(g-1)=(x_0g +g-1)+\sum_{j \neq 0} (x_jg)\mu^j.$$ 

\noindent Next, we see that $\dis \sup_{g  \in K^{\times}}v(\al^i g-1)=\sup_{g  \in A^{\times}}v(\al^i g-1)$. \\
For $g \in A^{\times}$, we have 
$$ v(\al^i g-1) = \min\{v(x_0g+g-1), v(x_1), \cdots, v(x_{p-1})\} \leq \min_{j \neq 0}\{v(x_j)\}.$$
 Thus, 
$\dis \sup_{g  \in A^{\times}}v(\al^i g-1) \leq \min_{j \neq 0}\{v(x_j)\}$. 
\\\\
Since $x_0 \in \m_A, x_0+1 \in A^{\times}$. Let $g:=(x_0+1)^{-1} \in A^{\times}$. 
\\Then $ x_0g+g-1=(x_0+1)g-1=0$ and the upper bound $\min_{j \neq 0}\{v(x_j)\}$ is attained. This contradicts the assumption on $i$ and concludes the proof.
\end{pf}

\noindent The following corollary was obtained in the proof of $(b) \implies (a)$.
\begin{sbcor}\label{bestkummer} Any $h$ satisfying (i) -- (v) is best.
\end{sbcor}
\section{Artin-Schreier extensions with higher rank valuations}\label{high}
Let the notation be as in \ref{N1}, \ref{sup}. For this section, we assume that $\ch K= \ch k = p$ and that $L/K$ is a non-trivial Artin-Schreier extension.
\\\\
In \cite{V1}, we assumed that $L/K$ is either defectless (with valuation of any rank) or has non-trivial defect and valuation of rank $1$. This restriction in the defect case is removed later in \cite{V2} while proving the mixed characteristic analogs.
In \cite[\S 8]{V2} we also remark that the results of \cite{V1} hold true for defect extensions with  valuations of a higher rank. This generalization follows once we prove \cite[Theorem 0.3]{V1} in this case. We will present the proof in \cref{pfhn}.

\subsection{Review of general definitions} 
We introduced the following definitions in \cite{V1}.
\begin{sbdefn}\label{js}
The ideal $\js$ of $B$ defined below 
generalizes the log Lefschetz number $j(\si)$ and is independent of the choice of the generator $\si$ of $\Gal(L/K)$.
$$
\js := \ \lb \left\{ \frac{\si(b)}{b}-1 \mid b \in L^{\times}
\right\} \rb. 
$$
\end{sbdefn}
\begin{sbdefn}
    
 The ideal $\sH$ of $A$ generalizes Kato's Swan conductor.
For an Artin-Schreier extension $L/K$, 
$$
\sH:= \ \lb \left\{ \frac{1}{h} \mid h   \in K^{\times}  \text{ and the solutions of
the equation } T^p-T=h \text{ generate } L/K \right\}
\rb.$$\end{sbdefn}

\noindent We note that the above $\sH$ was denoted by $H$ in \cite{V1} and its mixed characteristic analog was denoted by $\cH$  in \cite{V2}.\\

\begin{sbrem} While not apparent from the definition, $\sH$ is indeed an integral ideal. This is analogous to the fact that the classical Swan conductor is a non-negative integer.\end{sbrem}

\subsection{Trace}
The following lemma will be used in the proof of \Cref{hn}.
\begin{sblem}[See 6.3 \cite{KKS}]\label{trace}
Let $R$ be an integrally closed integral domain with the field
of fractions $F$. Let $E/F$ be a separable extension of fields of
degree $n$. Suppose that $\beta \in E$ is such that $E=F(\beta).$ 
Let $g(T)= \min_F(\beta)$, the minimal polynomial of $\beta$ over
$F$. Then

\begin{enumerate}
\item $\Tr_{E/F}\lb \frac{\beta^m}{g'(\beta)}\rb$ is zero for all $1
\leq m \leq n-2$ and $\Tr_{E/F}\lb \frac{\beta^{n-1}}{g'(\beta)}\rb
=1$.
\item Assume, in addition, that $\beta$ is integral over $R$. Then
$$ \{ x \in E \mid \Tr_{E/F}(xR[\beta]) \subset R \} =
\frac{1}{g'(\beta)}R[\beta].$$
\end{enumerate}
\end{sblem}

\subsection{Proof of \Cref{hn} for higher rank valuations}\label{pfhn}
 \noindent \Cref{hn} below is proved in \cite[\S 4]{V1} for defectless extensions without any assumptions on the rank of the valuation and for defect extensions assuming that the rank of the valuation is $1$. We now present a proof that also works for higher rank valuations. Its mixed characteristic analog is proved in \cite[\S 4]{V2} without any assumption on the rank of the valuation.
 
\begin{sbthm}\label{hn} We have
the following equality of ideals of $A$, generalizing  \ref{Kswan}.:
$$
  \sH =  \lb N_{L/K}(\js)\rb 
$$
\end{sbthm}

\begin{proof} 
If $L/K$ is defectless, then $\dis \js=\lb \frac{1}{\alpha} \rb$ and $\dis \sH=\lb \frac{1}{f} \rb$, where $f$ is best and $\alpha^p-\alpha=f$ generates $L/K$. Therefore, 
$\dis N_{L/K}(\js)=\lb N_{L/K}\lb\frac{1}{\alpha}\rb \rb=\sH$. (This was proved in \cite{V1}.)
\\\\
For the rest of the proof, assume that $L/K$ is a defect extension.\\
The inclusion $\sH \subset  \lb N_{L/K}(\js)\rb$ is already  proved in \cite{V1} without any restrictions on the rank of the valuation. We include the proof below for the sake of completeness.
\\\\
Let $h \in  K^{\times}$ such that $\alpha^p-\alpha=h$ and $L=K(\alpha)$. 
Then $N(\alpha)=h$. Since $\js$ is independent of choice of $\sigma$ (\ref{js}), we may choose $\sigma$ such that $\sigma(\alpha)=\alpha+1$, for the computation below.
We have  $$N\lb \frac{\sigma(\alpha)}{\alpha}-1\rb = N\lb \frac{\alpha+1}{\alpha}-1\rb = N\lb\frac{1}{\alpha}\rb=\dis \frac{1}{h} \in \lb N_{L/K}(\js)\rb.$$ 
\medskip

\noindent Next, we prove the reverse inclusion $\lb N_{L/K}(\js)\rb \subset \sH$. 
Given any $b  \in L^{\times} \backslash K$, we want to prove that $\dis N\lb\frac{\si(b)}{b}-1\rb \in \sH$. 
Since the defect is non-trivial, we have $e_{L/K}=1$, and hence, $L=B^{\times}K$. So it is enough to show this for $b  \in B^{\times} \backslash A$.
Let $v$ denote the valuations on both $L$ and $K$.
\\\\
Consider the formal expression $\dis \tr_{L/K}=\tr= \frac{\si^p-1}{\si-1}= (\si- 1)^{p-1}.$ \\ 
\\ Let $g(T)=\min_K(b)$. Define $$ \ga_b=\ga:=\lb \frac{b^{p-1}}{g'(b)} \rb \text{ and }  ~y_b=y:=  (\si- 1)^{p-2}(\ga). $$  Then $(\si-1)(y)=\tr(\ga)=1$ (by \cref{trace}), i.e., $\si (y)= y+1$.\\ Hence, $y$ satisfies the Artin-Schreier equation $y^p-y=N(y) \in K$ and  $L=K(y)$.\\\\\\
From the definition of $\sH$, it follows that $\dis \frac{1}{N(y)} =  N \lb\frac{\si(y)}{y}-1\rb \in \sH$.\\\\\\
It is now enough to show that $\dis \lb N\lb\frac{\si(b)}{b}-1\rb \rb \sub  \lb N \lb\frac{\si(y)}{y}-1\rb \rb.$
\\\\\\
Let us compare   $\dis v\lb\frac{\si(b)}{b}-1\rb=v(\si(b)-b)=:s$  and $\dis v\lb\frac{\si(y)}{ y}-1\rb =v\lb\frac{1}{ y}\rb =:s'$.\\

\noindent For $1 \leq i \leq p-1$, let
$\dis v\lb (\si- 1)^{i}(\ga)\rb =v\lb (\si- 1)^{i-1}(\ga)\rb +c_i$; 
$c_i \geq 0$. 
\\
Then we have the following.
\begin{enumerate}[(i)]
\item $v(\ga)=-(p-1)s\Rar \sum_{i=1}^{p-1}c_i=(p-1)s$.
\item $\dis\sum_{i=1}^{p-2}c_i=v(y)-v(\ga)=v(y)+(p-1)s$. 
\item  $c_{p-1}=v((\si-1)(y))-v(y)=-v(y)=s'$. 
\end{enumerate}

\noindent Consider $B_b:= A[ \{ \si^i(b) \mid 1 \leq i \leq p-1 \}] \subset B$. It is invariant under the action of $\si^j - 1$ for all $1 \leq j \leq p-1$. 
For the ideal ${\js}_{,b}
:= \lb \{\si(t)-t \mid t \in B_b \} \rb$ of $B_b,$ the ideal $ {\js}_{,b}B$ of $B$ is finitely generated and therefore, principal.\\\\
Observe that
$$\ga =\frac{b^{p-1} N(g'(b))/g'(b)}{N(g'(b))} \in \lb \frac{1}{N(g'(b))} \rb B_b.$$
\\\\
\noindent Therefore, 
 $c_i \geq s= v(\si(b)-b)$ for all $1 \leq i \leq p-2$.  \\\\Consequently, we have $s+v(y)=s-s'=s-c_{p-1} \geq 0$.\\\\
Since $\dis s \geq s'$, we get  $\dis \lb\frac{\si(b)}{b}-1\rb \sub  \lb\frac{\si(y)}{y}-1\rb$.\\\\\\ Hence, we have
$\dis \lb N\lb\frac{\si(b)}{b}-1\rb \rb \sub  \lb N \lb\frac{\si(y)}{y}-1\rb \rb.$

\end{proof}

\subsection*{Acknowledgments} The author is grateful to Dr. Mrudul Thatte for her unwavering support during difficult times.\\
The author is supported by UKRI grant MR/T041609/2.

\medskip

\noindent Vaidehee Thatte\\Department of Mathematics\\King's College London\\Strand, London WC2R 2LS, United Kingdom\\
Contact: \texttt{vaidehee.thatte@kcl.ac.uk}
 \\
%\texttt{https://sites.google.com/view/vaideheethatte}\\
\end{document}